\documentclass[a4paper,reqno]{amsart}

\usepackage[latin1]{inputenc}
\usepackage[T1]{fontenc}
\usepackage[english]{babel}
\usepackage{amsfonts,amsmath,amssymb,amsthm}
\usepackage{mathrsfs}
\usepackage{graphicx}
\usepackage{color}
\usepackage{enumerate}

\newtheorem{theorem}{Theorem}[section]
\newtheorem{proposition}[theorem]{Proposition}
\newtheorem{lemma}[theorem]{Lemma}

\theoremstyle{definition}

\theoremstyle{remark}

\numberwithin{equation}{section}

\renewcommand{\P}{\mathbb P}
\newcommand{\N}{\mathbb N}
\newcommand{\E}{\mathbb E}
\newcommand{\R}{\mathbb R} 
\newcommand{\dd}{\mathrm d}
\newcommand{\eps}{\varepsilon}
\newcommand{\e}{\mathrm e}

\newcommand{\vect}{\boldsymbol}

\def\1{{\mathchoice {1\mskip-4mu\mathrm l}      
{1\mskip-4mu\mathrm l}
{1\mskip-4.5mu\mathrm l} {1\mskip-5mu\mathrm l}}}

\begin{document}

\title[Large deviations for the maximum  of Weibull variables]{Large deviations for the maximum and reversed order statistics of Weibull-like variables}

\date{27 May 2024} 
\author{Sabine Jansen}

\address{Mathematisches Institut, LMU Munich, Theresienstr. 39, 80333 Munich, Germany}
\email{jansen@math.lmu.de} 

\begin{abstract}
	Motivated by metastability in the zero-range process, we consider i.i.d.\ random variables with values in $\N_0$ and Weibull-like (stretched exponential) law $\P(X_i =k) = c \exp( - k^\alpha)$, $\alpha \in (0,1)$. We condition on large values of the  sum $S_n= \mu n + s n^\gamma$ and prove large deviation principles for the rescaled maximum $M_n /n^\gamma$ and for the reversed order statistics. The scale is $n^\gamma$ with $\gamma = 1/(2-\alpha)$; on that scale, the big-jump principle for heavy-tailed variables and a naive normal approximation for moderate deviations yield bounds of the same order $n^{\gamma \alpha} = n^{2\gamma-1}$, the speed of the large deviation principles.  The rate function for $M_n/n^\gamma$ is non-convex and solves a recursive equation similar to a Bellman equation. \\
	
	
\end{abstract}

\maketitle

\section{Introduction} 

Fix $\alpha\in (0,1)$ and let $c = 1/\sum_{k\in \N_0} \exp(-k^\alpha)$. 
Consider $n$ i.i.d., integer-valued random variables $X_1,X_2,\ldots$ with stretched exponential law
\begin{equation} \label{eq:stretched}
	\P(X_i = k) = c \exp( - k^\alpha) \qquad (k\in \N_0),
\end{equation}
expected value $\mu = \E[X_i]$ and variance $\sigma^2$. 
Set $\gamma = 1/(2-\alpha)$. 
It is known that the partial sums $S_n= X_1+\cdots + X_n$ satisfy, for every $s>0$,
\begin{equation} \label{eq:sumasympt}
	\lim_{n\to \infty} \frac{1}{n^{\gamma\alpha}} \log \P( S_n = \lfloor \mu n + s n^\gamma\rfloor ) = - \min_{x\in [0,s]} \Bigl( x^\alpha + \frac{1}{2\sigma^2} (s-x)\Bigr)^2.
\end{equation}
Eq.~\eqref{eq:sumasympt} follows from results by Nagaev \cite{nagaev1968}, see also Armend{\'a}riz, Grosskinsky, Loulakis \cite{armendariz-grosskinsky-loulakis2013},  Ercolani, Jansen and Ueltschi \cite{ercolani-jansen-ueltschi2019} and Brosset, Klein, Lagnoux and Petit~\cite{brosset-klein-lagnoux-petit2022}.  For small $s$ the minimum on the right side is attained at $x=0$, for large $s$ the origin $x=0$ is only a local minimizer and the global minimum is attained at some larger $x>0$. The two different global minimizers correspond to two different behaviors for the maximum $M_n = \max(X_1,\ldots,X_n)$ under the conditioned law $\P(\cdot \mid S_n = \lfloor \mu n + s n^\gamma\rfloor)$. For small $s$, $M_n/n^\gamma\to 0$ in law  and the maximum is typically of the order of some power of $\log n$ \cite[Theorem 2.3]{armendariz-grosskinsky-loulakis2013}; for large $s$, $M_n/n^\gamma$ converges instead to the non-zero minimizer in the right side of Eq.~\eqref{eq:sumasympt}. 

These considerations suggest an underlying large deviation principle for the distribution of $M_n/n^\gamma$ under the condition $S_n = \lfloor \mu n + s n^\gamma\rfloor$. A natural guess is that the rate function should be, up to an additive constant, precisely the function minimized on the right side of~\eqref{eq:sumasympt}. It turns out that this is not true for all $s$. Our main result is a large deviation principle for the rescaled maximum in which the rate function solves a recursive equation (Eq.~\eqref{eq:inf-convolution} and Theorem~\ref{thm:maxldp}), and we prove that in general the rate function is different from the natural guess (Theorem \ref{thm:dyntrans}). In addition, we prove a large deviation principle for the full sequence of reversed order statistics (maximum, second largest variable, third largest, etc.; Theorem~\ref{thm:main}). 

From the point of view of the theory of heavy-tailed variables \cite{denisov-dieker-shneer2008}, our results are a nice but, perhaps, somewhat unessential addendum. In fact our proofs are fairly elementary. The relevance of our results instead comes from interacting particle systems, more precisely metastability for the zero-range process. 

The zero-range process on the lattice $\{1,\ldots,n\}$ is a continuous-time Markov process $(\eta(t))_{t\geq 0}$ with state space $\N_0^n$ \cite{spitzer1970}. The configuration at time $t$ is denoted $(\eta_x(t))_{x=1,\ldots,n}$.  
Particles may hop from site $x$ to site $y$, thus the possible transitions are from $\eta$ to $\eta^{x,y}$ where $\eta_z^{x,y} = \eta_z - \delta_{x,z} + \delta_{y,z}$. Hopping probabilities are governed by a stochastic matrix $(p(x,y))_{1\leq x,y\leq n}$ with zeros on the diagonal  and by rates $g(n)$, $n\in \N_0$, with $g(0)=0$. The generator is 
\[
	Lf (\eta) = \sum_{\substack{1\leq x,y\leq n:\\ x\neq y}}  g(\eta_x) p(x,y) \bigl( f( \eta^{x,y}) - f(\eta)\bigr).
\] 
We assume $p(x,y) = p(y,x)$ for all $x,y$ and choose rates
\[
	g(n) = \frac{\exp( - (n-1)^\alpha)}{\exp( - n^\alpha)} =  1+ \frac{\alpha}{n^{1-\alpha}} + o\Bigl( \frac{1}{n^{1-\alpha}}\Bigr).
\] 
When started in an $N$-particle configuration ($\sum_{x=1}^n \eta_x = N$) with fixed $N\in \N$, in the long run the process converges in law to  an equilibrium distribution that is exactly the law $P_{n,N}$ of $n$ stretched exponential variables conditioned on $X_1+\cdots + X_n = N$. 

If the system size $n$ and particle excess $N- \mu n$ are large, the equilibrium measure gives high probability to configurations in which one of the $n$ sites swallows of all the excess $\eta_x \approx N - \mu n$, the \emph{condensate}, while all other sites have small occupation numbers \cite{armendariz-loulakis2009, grosskinsky-schuetz-spohn2003}. This raises the following question: If we start in a configuration with a large excess $N-\mu n$ evenly spread over the $n$ sites, how long does it take until a condensate---a site with $\eta_x \approx N-\mu n$--- is formed? 

In a \emph{metastable} regime one expects a large transition time, ideally described with some barrier in an energy landscape \cite{berglund2013,bovier-denhollander-book}. Many results in this direction are available for other models, e.g., for dynamic Ising or Curie-Weiss models \cite{bovier-denhollander-book}. For the zero-range process, however, most metastability results have focused on the dynamics of the condensate once it has formed---how long does it take for the condensate to move from one site to another?  See, e.g., Beltr{\'a}n and Landim \cite{beltran-landim2012} and  Bovier and den Hollander \cite[Chapter 21]{bovier-denhollander-book}. An exception is Chleboun and Grosskinsky \cite{chleboun-grosskinsky2010} who discuss metastability on the same scale as ours, but without a theorem on the asymptotics of hitting times. 

Our results pave the way for addressing the open question of metastable condensation time. For large $s$, the rate function for the maximum $M_n /n^\gamma$ is non-convex, it has a local and a global minimizer separated by an energy barrier, and one may hope to set the potential-theoretic machinery for metastability to work. 

\section{Main results} 

Let $X_1,X_2,\ldots$ be i.i.d.\ stretched exponential, $\N_0$-valued variables with law $\P(X_i = k) = c\exp( - k^\alpha)$ as in~\eqref{eq:stretched}. The expected value, variance and partial sums are denoted $\mu$, $\sigma^2$, and $S_n$. The maximum is 
\[
	M_n = \max(X_1,\ldots, X_n).
\]
Set 
\[
	\gamma:= \frac{1}{2-\alpha}.
\]
Let $\mathcal D\subset \R_+^\N$ be the space of decreasing sequences $x_1\geq x_2\geq \cdots$. Given $s\in \R$, define $F_s:\mathcal D\to \R_+\cup \{\infty\}$ by
\[
	F_s(x_1,x_2,\ldots) = 
			\sum_{j=1}^\infty x_j^\alpha + \frac{1}{2\sigma^2}(s- \sum_{j=1}^\infty x_j)^2
\] 
and $f_s: \R_+ \to \R$ by 
\[
	f_s(y):= \inf \Bigl\{ F_s(x_1,x_2,\ldots):\, (x_j)_{j\in \N}\in \mathcal D,\ x_1 = y\Bigr\}.
\] 
Notice 
\begin{equation} \label{eq:inf-convolution}
	f_s(y) = y^\alpha + \inf_{z\in [0,y]} f_{s-y} (z). 
\end{equation} 
This equation is similar to Eq.~(6) in Chleboun and Grosskinsky~\cite{chleboun-grosskinsky2015} and to  Bellman equations from dynamic programming. 

In the following $(s_n)_{n\in \N}$ is a sequence in $\R$ such that $\mu n + s_n n^\gamma$ is in $\N_0$, for all $n\in \N$, and $\lim_{n\to \infty} s_n = s$. 

\begin{theorem}\label{thm:maxldp} 
	Let $Q_{n,s}$ be the distributions of $M_n/n^\gamma$ under the conditioned law $\P(\cdot \mid S_n =  \mu  n + s_n n^\gamma)$. Then $(Q_{n,s_n})_{n\in\N}$ satisfy the large deviation principle with speed $n^{\gamma\alpha}$ and good rate function $f_s - \inf_{\R_+} f_s$. 
\end{theorem}

Thus $f_s$ is lower semicontinuous with compact level sets and for every closed set $A\subset \R_+$
\[
	\limsup_{n\to \infty} \frac{1}{n^{\gamma\alpha}} \log \P\Bigl(\frac{M_n}{n^\gamma} \in A\mid S_n = \mu n + s_n n ^\gamma\Bigr)\leq - \inf_{y\in A} \bigl( f_s(y) - \inf_{\R_+} f_s\bigr)
\] 
and for every open set $O \subset \R_+$
\[
	\liminf_{n\to \infty} \frac{1}{n^{\gamma\alpha}} \log \P\Bigl(\frac{M_n}{n^\gamma} \in O\mid S_n = \mu n + s_n n ^\gamma\Bigr) \geq - \inf_{y\in O} \bigl( f_s(y) - \inf_{\R_+} f_s\bigr).
\] 
Theorem~\ref{thm:maxldp} follows, by the contraction principle \cite[Chapter 4.2]{dembo-zeitouni-book}, from the large deviation principle for the full reversed order statistics. Rearranging $X_1,\ldots, X_n$ in decreasing order, we obtain a vector $(X_{[1:n]}, \ldots, X_{[n:n]})$ with 
\[
	X_{[1:n]} \geq X_{[2:n]} \geq \cdots \geq X_{[n:n]}.
\]
In particular, $X_{[1:n]} = \max (X_1,\ldots, X_n)=M_n$. In addition we define $X_{[j:n]} =0$ for $j\geq n+1$. We equip the space $\mathcal D$ of decreasing sequences with  the trace of the product topology.

\begin{theorem} \label{thm:main}
	Let $\mathbb Q_{n,s_n}$ be the distribution of $(\frac{1}{n^\gamma} X_{[j:n]})_{j\in\N}$ under the conditioned law $\P(\cdot \mid S_n = \mu  n + s_nn^\gamma)$. Then $(\mathbb Q_{n,s_n})_{n\in \N}$ satisfies the large deviation principle withs speed $n^{\gamma\alpha}$ and  good rate function $I_s = F_s - \inf_\mathcal D F_s$. 
\end{theorem}

The theorem is proven at the end of Section~\ref{sec:ub}. It says that
$I_s$ is lower semi-continuous with compact sublevel sets, and for every closed set $\mathcal A\subset \mathcal D$ (trace of the product topology) 
\begin{equation} \label{eq:ldp-ub}
	\limsup_{n\to \infty}\frac{1}{n^{\gamma\alpha}} \log \P\Biggl( \Bigl(\frac{X_{[j:n]}}{n^\gamma}\Bigr)_{j\in \N} \in \mathcal A\, \Big\vert\,S_n =  \mu n + s_n n^\gamma  \Biggr) \leq - \inf_{\mathcal A} I_s
\end{equation} 
and for every open set $\mathcal O\subset \R_+^\N$, 
\begin{equation} \label{eq:ldp-lb}
	\liminf_{n\to \infty}\frac{1}{n^{\gamma\alpha}} \log \P\Biggl(  \Bigl(\frac{X_{[j:n]}}{n^\gamma}\Bigr)_{j\in \N} \in \mathcal O\, \Big\vert\, S_n =  \mu n + s_n n^\gamma  \Biggr) \geq - \inf_{\mathcal O} I_s.
\end{equation} 
Turning back to the large deviation principle for the rescaled maximum $M_n /n^\gamma$, let us have a closer look at the rate function $f_s - \inf_{\R_+} f_s$. As noted in the introduction, the function $f_s$ should be  contrasted with the simpler function
\[
	g_s(y):= y^\alpha + \frac{1}{2\sigma^2} (s-y)^2
\] 
that appears on the right side of~\eqref{eq:sumasympt}. Define a threshold 
\[
	s_1:= \frac 1\gamma (\sigma^2)^\gamma (2-2\alpha)^{\gamma-1}. 
\] 
Then $y=0$ is a global minimizer for $g_s(y)$ if and only if $s\leq s_1$; for $s>s_1$ it is merely a local minimizer (see Appendix~\ref{app:variational}). Clearly 
\[
	f_s(y)\leq  F_s(y,0,0,\ldots)= g_s(y) 
\]
for all $s,y$. The next theorem says that there is an additional threshold $s_2>s_1$ beyond which $f_s$ and $g_s$ do not coincide. (They coincide for small and large $y$ but there is a middle range in which $f_s(y) < g_s(y)$.)

\begin{theorem} \label{thm:dyntrans}
	There exists $s_2\in (s_1,\infty)$ such that the following holds true: 
	\begin{enumerate} 
		\item [(a)] For $s\leq s_2$, we have $f_s(y) = g_s(y)$ for all $y\geq 0$. 
		\item [(b)] For $s >s_2$,
		the set $J_s = \{y: f_s(y) < g_s(y)\}$ is non-empty, and contained in some interval $(a_s,b_s)$ with $0< a_s<b_s< \infty$.
	\end{enumerate} 
\end{theorem}

The theorem is proven in Section~\ref{sec:dyntrans}. 
The additional treshold $s_2$ is related to the dynamical transition identified by Chleboun and Grosskinsky for a zero-range process with size-dependent rates \cite{chleboun-grosskinsky2015}. Chleboun and Grosskinsky derive a free energy landscape for the maximum and second largest  occupation number, similar to ours. They study the motion of the condensate and identify two possible mechanisms: either the condensate dissolves and is formed anew at another site, or a second condensate is grown in a neighboring site and eventually absorbs the original condensate. The density $\rho_{\mathrm{dyn}}$ beyond which the second mechanism is more likely is identified with formulas similar to the recursive equation~\eqref{eq:inf-convolution}.

\section{Proofs} 

\subsection{Lower semi-continuity} 

The rate function in a large deviation principle is \emph{good} if it is lower semicontinuous with compact level sets. The goodness of our rate functions is a consequence of the following lemma. 

\begin{lemma} \label{lem:lsc} 
	The function $F_s$ is lower semicontinuous on $\mathcal D$ and it has compact level sets. 
\end{lemma} 

\begin{proof} 
Compactness of level sets is easy: For $M\in (0,\infty)$, the level set $\{\vect x:\, F_s(\vect x)\}$ is contained in $\mathcal D\cap [0,M^{1/\alpha}]^\N$, which is compact in the product topology. 

For the lower semicontinuity, let $(\vect x^{(m)})_{m\in \N}$ be a sequence in $\R_+^\N$ that converges in the product topology to some limiting $\vect x\in \R_+^\N$. Thus, writing $\vect x^{(m)} = (x_j^{(m)})_{m\in \N}$ and $\vect x=(x_j)_{j\in \N}$, we have $x_j^{(m)} \to x_j$ for all $j\in \N$. We have to show that
\[
	F_s(x_1,x_2,\ldots) \leq \liminf_{m\to \infty} F_s(x_1^{(m)}, x_2^{(m)},\ldots).
\] 
When the right side is infinite, the inequality is trivial. When the right side is finite, we may assume that 
\[
	C:= \limsup_{m\to \infty} F_s(x_1^{(m)}, x_2^{(m)},\ldots)<\infty
\] 
(pass to a subsequence if needed). By Fatou's lemma,
\[
	\sum_{j=1}^\infty x_j ^\alpha \leq \liminf_{m\to \infty} \sum_{j=1}^\infty (x_j^{(m)})^\alpha.
\] 
Next notice that for every $\ell \in \N$, 
\begin{equation} \label{eq:sosplit}
	\Bigl| \sum_{j=1}^\infty x_j^{(m)} - \sum_{j=1}^\infty x_j\Bigr| 
	\leq \sum_{j=1}^\ell|x_j^{(m)} -x_j| + \sum_{j>\ell} (x_j+ x_j^{(m)})
\end{equation}
and 
\[
	\sum_{j>\ell} x_j \leq x_\ell^{1-\alpha} \sum_{j>\ell} x_j^\alpha \leq C x_\ell^{1-\alpha}.
\] 
The sum  $\sum_{j>\ell} x_j^{(m)}$ is bounded in a similar way.  We insert these bounds into~\eqref{eq:sosplit} and pass to the limit $m\to \infty$ at fixed $\ell$. This gives 
\begin{equation}\label{eq:sosplit2}
	\limsup_{m\to \infty} 	\Bigl| \sum_{j=1}^\infty x_j^{(m)} - \sum_{j=1}^\infty x_j\Bigr| \leq 2 C x_\ell^{1-\alpha}.
\end{equation}
As $\sum_j x_j^\alpha \leq C<\infty$, we must have $x_\ell \to 0$ as $\ell\to \infty$. We take the limit $\ell \to\infty$ in~\eqref{eq:sosplit2} and find that $\sum_{j=1}^\infty x_j^{(m)} \to \sum_{j=1}^\infty x_j$. The lower semicontinuity of $F_s$ follows. 
\end{proof}

\subsection{Lemmas on truncation errors} 

Stretched exponential random variables are heavy-tailed: the moment generating function $\E[\exp( t X_i)]$ is infinite for all $t>0$.
Therefore, it is convenient to work with truncated generating functions. 

\begin{lemma}\label{lem:truncation1}
 Let $c_1>0$ and $t_n = c_1 n^{\gamma-1}$. Then, for every $c_2<c_1^{- 1/(1-\alpha)}$ and setting $m_n= \lfloor c_2 n^\gamma\rfloor $, we have for all $k\in \N$
	\[
		\sup_{n\in \N} \sum_{j=1}^{m_n} j^k \exp( t_n j - j^\alpha) <\infty. 
	\] 
\end{lemma} 

\begin{proof}
	For $\delta\in (0,1)$, we have $t_n j \leq (1-\delta) j^\alpha$ if and only if 
	$(1-\delta)^{-1} c_1 n^{\gamma-1} \leq j^{\alpha-1}$ i.e.\ $(1-\delta) c_1^{-1} n^{1-\gamma} \geq j^{1-\alpha}$ hence 
	\[
		j\leq \Bigl(\frac{1-\delta}{c_1}\Bigr)^{1/(1-\alpha)} n^{(1-\gamma)/(1-\alpha)}.
	\] 
	In view of
	\[
		\frac{1-\gamma}{1-\alpha} = \frac{1}{1-\alpha}\Bigl( 1 - \frac{1}{2-\alpha}\Bigr) = \frac{1}{1-\alpha} \frac{1-\alpha}{2-\alpha} = \gamma,
	\] 
   the	condition becomes 
	\[
		m_n \leq \Bigl(\frac{1-\delta}{c_1}\Bigr)^{1/(1-\alpha)} n^\gamma. 
	\] 
	Given $c_2 < c_1^{-1/(1-\alpha)}$, we can find $\delta>0$ such that the above inequality holds true, and we get
	\[
		\sum_{j=1}^{m_n} j^k \exp( t_n j - j^\alpha)  \leq \sum_{j=1}^\infty j^k \exp( - \delta j^\alpha) =: C_k(\delta)<\infty.  \qedhere
	\] 
\end{proof} 

\begin{lemma} \label{lem:truncation2}
	For every $k>0$, as $\ell\to \infty$, 
	\[
		\sum_{j>\ell} j^k \exp(-j^\alpha) = O\Bigl( \ell^{k+1-\alpha} \exp(- \ell^\alpha)\Bigr).
	\] 
\end{lemma} 

\begin{proof}
	The function $\R_+\ni x\mapsto x^k \exp(-x^\alpha)$ is eventually decreasing, hence we may bound sums by integrals and find for large $\ell$
	\[
		\sum_{j>\ell} j^k \exp( - j^\alpha)  \leq \int_\ell^\infty u^k \exp( - u^\alpha) \dd u = \frac{1}{\alpha}\int_{\ell^{\alpha}}^\infty t^{(k+1)/\alpha - 1} \e^{-t} \dd t. 
	\]
	On the right side we recognize an incomplete Gamma function. It is known that for every $s>0$, as $x\to \infty$, 
	\[
		\int_x^\infty t^{s-1} \e^{-t} \dd t= \bigl(1+o(1)\bigr) x^{s-1} \e^{-x}. 
	\]
	The lemma follows. 
\end{proof}

\subsection{Normal approximation} 

When we condition on the maximum to be small enough, the normal approximation to moderate deviation events is justified. This is the key ingredient to our proofs. In this subsection we do not impose $s_n\to s$ and merely ask that $(s_n)$ is bounded and $\mu n + s_n n^\gamma \in \N_0$ for all $n$. 

\begin{proposition}\label{prop:normal}
	Let $[a,b]$ and $[\eps,\kappa]$ be compact intervals with $a < 0 < b$, $0< \eps < \kappa$ and $\kappa < (\sigma^2/b)^{-1/(1-\alpha)}$. 
	Further let $(s_n)_{n\in \N}$ be a sequence in $[a,b]$ and $(\kappa_n)_{n\in \N}$ a sequence in $[\eps,\kappa]$. Then 
	\[
		\Bigl| \log \P( M_n \leq \kappa_n n^\gamma, S_n = \mu n + s_n n^\gamma) + \frac{s_n^2}{2\sigma^2} n^{2\gamma-1}\Bigr| \leq C n^{3\gamma - 2}
	\] 	
	for all $n\in \N$ and some constant $C= C(a,b,\eps,\kappa)$. 
\end{proposition}

For the proof, set $m_n:= \lfloor \kappa_n n^\gamma\rfloor $ and 
define the truncated cumulant generating function 
\[
	\varphi_n(t) =\log \Bigl( \sum_{j=0}^{m_n} c \exp( - j^\alpha) \e^{tj}\Bigr)
\]
so that
\begin{equation} \label{eq:markovub}
	\P(S_n = \mu n + s_nn^\gamma, M_n \leq \kappa_n n^\gamma) 
	\leq \inf_{t\in \R} \exp\Bigl( n \varphi_n(t) - (\mu n+ s_n n^\gamma) t\Bigr). 
\end{equation} 
As $\varphi_n'$ is continuous and strictly increasing at goes to $0$ as $t\to -\infty$ and to $\infty$ as $t\to \infty$, the equation $\varphi'_n(t_n) = \mu + s_n n^{\gamma-1}$ has a unique solution $t_n\in \R$. 

\begin{lemma}\label{lem:tnsol}
	Under the conditions of Proposition~\ref{prop:normal}: The unique solution $t_n$ to $\varphi'_n(t_n) = \mu + s_n n^{\gamma-1}$ satisfies			\[
		\Bigl| \varphi_n(t_n) - (\mu + s_n n^{\gamma-1}) t_n + \frac{s_n^2}{2\sigma^2}\, n^{2\gamma-2}\Bigr|\leq C n^{3\gamma-3}
	\]
	for all $n\in \N$ and some $ C= C(\eps,\kappa,a,b)$. 
\end{lemma} 

\begin{proof}
	Let $r_n: =\sum_{j> m_n} j^3 \exp( - j^\alpha)$.	
	By Lemma~\ref{lem:truncation2}, we know that 
	\[
		r_n \leq \sum_{j> \eps n^\gamma} j^3 \exp(-j^\alpha) = O\Bigl( \exp( - \mathrm{const} n^{\gamma\alpha})\Bigr)
	\] 	
	as $n\to \infty$. The bound is uniform on $m_n = \lfloor \kappa_n n^\gamma \rfloor \geq \eps n^\gamma$.	We note 
	\[
		\varphi'_n(0) = \frac{\sum_{j=1}^\infty j\exp( - j^\alpha)+ O(r_n)}{\sum_{j=1}^\infty \exp( -j ^\alpha) + O(r_n)} = \mu + O(r_n), \quad \varphi''_n(0) = \sigma^2 + O(r_n).
	\] 
Next let $c_1> b/\sigma^2$ with $\kappa < c_1^{-1/(1-\alpha)}$ and $\tau_n:= c_1 n^{\gamma-1}$. We have 
	\[
		\varphi'_n(\tau_n) = \varphi'_n(0) + \tau_n \varphi_n''(0) + \frac{\tau_n^2}{2} \varphi_n'''(\tau'_n)
	\]
	for some $\tau'_n \in [0,\tau_n]$. By Lemma~\ref{lem:truncation1}, the third derivative $\varphi_n'''$ is bounded on $[0,\tau_n]$, uniformly in $n$. Thus, 
	\begin{align*}
		\varphi'_n(\tau_n) & = \mu +O(r_n) + \tau_n\bigl(\sigma^2+O(r_n)\bigr) + O(\tau_n^2) \\
		& = \mu + c_1 \sigma^2 n^{\gamma-1} + O(n^{2\gamma-2}),
	\end{align*}
	hence $\varphi'_n(\tau_n)> \mu + s_n n^{\gamma-1}$ for sufficiently large $n$ and $t_n < \tau_n$.  For the asymptotic behavior of $t_n$, we note
	\[
		\mu + s_n n^{\gamma-1} = \varphi'_n(t_n) = \mu + \sigma^2 t_n + O(n^{2\gamma-2})
	\]
	which gives $t_n = s_n n^{\gamma-1} /\sigma^2 + O(n^{2\gamma-2})$. Finally, 
	\[
		\varphi_n(t_n) - (\mu + s_n n^{\gamma-1}) t_n
		 = \frac12 \sigma^2 t_n^2 + O(r_n) + O(t_n^3) = \frac{s^2}{2\sigma^2} n^{2\gamma-2} + O( n^{3\gamma-3}).
	\] 
	The bounds depend only on the remainder term $r_n$ and on bounds on derivatives; they are uniform for $s_n\in [a,b]$, $\kappa_n \in [\eps,\kappa]$. 
\end{proof}

\begin{proof} [Proof of Proposition~\ref{prop:normal}] 
	The inequality~\eqref{eq:markovub} and Lemma~\ref{lem:tnsol} yield right away the upper bound
	\[
		\P(S_n = \mu n + s_n n^{\gamma}, M_n\leq \kappa_n n^{\gamma}) 
		\leq \exp\Bigl( - \frac{s_n^2}{2 \sigma^2} n^{2\gamma-1} + O(n^{3\gamma-2})\Bigr). 
	\] 	
	For the lower bound, let $\widehat X^{(n)}_i$, $i\in \N$, be i.i.d.\ random variables with law 
	\[
		\P(\widehat X_i^{(n)} = k) = \1_{\{ k \leq \kappa n^\gamma\}} \frac{ \exp(t_n k - k^\alpha)}{\sum_{j\leq \kappa_n n^\gamma} \exp( t_n j- j^\alpha)} \quad (k\in \N_0)
	\]
	and let $\widehat S_n =\widehat X^{(n)}_1+\cdots+ \widehat X^{(n)}_n$. Then 
	\begin{multline*}
		\P(S_n = \mu n + s_n n^{\gamma}, M_n\leq \kappa_n n^{\gamma})\\ 
			= \exp\Bigl( n\varphi_n(t_n) - (\mu + s_n n ^\gamma) t_n\Bigr)
			\P( \widehat S_n = \mu n + s_n n^\gamma).	
	\end{multline*} 
	The asymptotic behavior of the exponential is given by Lemma~\ref{lem:tnsol}. The variables $\widehat X^{(n)}_i$ have expected value $\mu + s_n n^{\gamma-1}$, variance $\sigma^2+ O(n^{\gamma-1})$, and their third moments are bounded, uniformly in $n$. The local central limit theorem for integer-valued random variables \cite[Chapter 4]{ibragimov-linnik-book} yields 
	\[
		\P\Bigl( \widehat S_n = \mu n + s_n n^\gamma\bigr) 
		= \frac{1}{\sqrt{2\pi n (\sigma^2 +O(n^{\gamma-1}))}} + O( n^{-3/2})
 = \exp\Bigl( O(\log n)\Bigr). 
	\]
	Carefully retracing the proof of the local limit theorem, one finds that the normal approximation is uniform in the range considered here; we skip the details as we shall actually only need the uniformity for the upper bound. 
\end{proof}

\subsection{Lower bound}


As a preparation for the proof of Theorem~\ref{thm:maxldp}, we prove the corresponding lower bound for non-normalized measures $A\mapsto \P(A\cap \{S_n = \mu n + s_n n^\gamma\})$. 

\begin{proposition} \label{prop:lb}
	Let $s\in \R$ and let $(s_n)_{n\in \N}$ be a sequence with $s_n\to s$ and $\mu n + s_n n^\gamma \in \N_0$ for all $n$. Then, for every open set $\mathcal O\subset \mathcal D$, 
	\[
		\liminf_{n\to \infty} \frac{1}{n^{\gamma\alpha}} \log \P\Bigl( \Bigl( \frac{X_{[j:n]}}{n^\gamma} \Bigr)_{j\in \N} \in \mathcal O,\ S_n = \mu n + s_n n^\gamma \Bigr) \geq - \inf_{\mathcal O}F_s.
	\] 
\end{proposition} 

\begin{proof} 
 The only interesting case is when $\inf_\mathcal O F_s <\infty$. As $F_s$ is lower semi-continuous and has compact level sets, the rate function has a minimizer on $\mathcal O$; let $\vect x = (x_j)_{j\in \N}\in \mathcal O$ be such that $F_s(\vect x) = \inf_\mathcal O F_s<\infty$.

The product topology is metrizable with metric 
\begin{equation} \label{eq:metric}
	d(\vect x, \vect y) = \sum_{j=1}^\infty \frac{1}{2^j} \frac{|x_j-y_j|}{1+|x_j-y_j|}.
\end{equation} 
As $\mathcal O$ is open, there exists $\eps>0$ such that the open ball with radius $\eps$ centered at $\vect x$ is contained in $\mathcal O$. 
A sufficient condition for $d(\vect x,\vect y) \leq \eps$ is 
\[
	\sum_{j=1}^\ell |y_j-x_j| \leq \frac{\eps}{2},\quad \frac{1}{2^\ell} \leq \frac \eps 2. 
\] 
Let $\ell\in \N$ large enough so that $2^\ell>2/\eps$ and pick $\delta < \eps /\ell$. Thens, for sufficiently large $n$, 
\begin{align*} 
	&\P\Bigl( \Bigl( \frac{X_{[j:n]}}{n^\gamma}\Bigr)_{j\in \N}\in \mathcal O, S_n = \mu n + s_nn^\gamma \Bigr) \\
	&\quad \geq \P\Bigl( \forall j\leq \ell:\, X_j = \lfloor (x_j+\delta) n^\gamma\rfloor, X_{[\ell+1:n]}\leq \lfloor (x_\ell +\delta) \rfloor,\, S_n = \mu n + s_n n^\gamma\Bigr)\\
	&\quad = \P\Bigl( \forall j\leq \ell:\, X_j = \lfloor (x_j+\delta) n^\gamma\rfloor\Bigr)\\
	&\qquad \qquad \times \P\Bigl(S_{n-\ell} = \mu n + (s_n- \sum_{j=1}^\ell \lfloor (x_j+\delta) \rfloor n^\gamma), M_{n-\ell} \leq \lfloor (x_\ell+\delta) n^\gamma\rfloor \Bigr).
\end{align*} 
The last probability is evaluated with Proposition~\ref{prop:normal}. We take the logarithm, divide by $n$, and let first $n\to \infty$ at fixed $\delta$ and then $\delta\downarrow 0$. This gives 
\begin{multline*}
	\liminf_{n\to \infty}\frac{1}{n}\log \P\Bigl( \Bigl( \frac{X_{[j:n]}}{n^\gamma}\Bigr)_{j\in \N}\in \mathcal O, S_n = \mu n + s_n n^\gamma \Bigr)\\
	\geq - \sum_{j=1}^\ell x_j^\alpha - \frac{1}{2\sigma^2}\Bigl( s - \sum_{j=1}^\ell x_j\Bigr)^2. 	
\end{multline*} 
To conclude, let $\ell\to \infty$. The right side above becomes $- F_s(\vect x)$, which is $- \inf_\mathcal O F_s$ by the choice of $\vect x$. 
\end{proof} 

\subsection{Upper bound} \label{sec:ub}

Next we turn to the upper bound for the non-normalized measures.

\begin{proposition} \label{prop:ub}
	Let $s\in \R$ and let $(s_n)_{n\in \N}$ be a sequence with $s_n\to s$ and $\mu n + s_n n^\gamma \in \N_0$ for all $n$. Then, for every closed set $\mathcal A\subset \mathcal D$, 
	\[
		\limsup_{n\to \infty} \frac{1}{n^{\gamma\alpha}} \log  \P\Bigl( \Bigl( \frac{X_{[j:n]}}{n^\gamma} \Bigr)_{j\in \N} \in \mathcal A,\ S_n = \mu n + s_n n^\gamma \Bigr) \leq - \inf_{\mathcal A}F_s. 
	\] 
\end{proposition}

For the proof, fix $\kappa < (\sigma^2/s)^{-1/(1-\alpha)}$ if $s>0$ and $\kappa >0$ arbitrary if $s=0$.  For $m>0$ and $\ell\in \N$, further let
\[
	\mathcal A_{m,\ell} = \mathcal A\cap \{\vect x:\, x_1\leq m\}\cap \{\vect x:\, x_{\ell+1}\leq \kappa\bigr\}. 
\] 
We split
\begin{multline} \label{eq:A3}
	\P\Bigl( \Bigl( \frac 1{n^\gamma} X_{[j:n]}\Bigr)_{j\in \N}\in \mathcal A, S_n = \mu n + s_n n^\gamma \Bigr) \\ 
	\leq \P\Bigl( \Bigl( \frac 1{n^\gamma} X_{[j:n]}\Bigr)_{j\in \N}\in \mathcal A_{m,\ell}, S_n =  \mu n + s_n n^\gamma  \Bigr)\\ 
		+ \P\bigl( M_n > m n^\gamma\bigr) + \P\bigl( X_{[1:n]}\geq \cdots \geq X_{[\ell:n]} > \kappa n^\gamma \bigr).
\end{multline} 
The last term is easily bounded: we have 
\begin{align*}
	 \P\Bigl(\forall j\leq \ell:\, X_{[j:n]} > \kappa n^\gamma \Bigr)
	& \leq \binom{n}{\ell} \Bigl( \sum_{j >n^\gamma \kappa} c\exp( -j^\alpha) \Bigr)^\ell \\
	& \leq \binom{n}{\ell} \Bigl( \frac 1\alpha (1+o(1)) (n^\gamma \kappa)^{1-\alpha} \exp( - n^{\gamma\alpha} \kappa^\alpha)\Bigr)^\ell
\end{align*}
(see Lemma~\ref{lem:truncation2}), hence 
\begin{equation}\label{eq:ub1}
	\limsup_{n\to \infty} \frac 1{n^{\gamma\alpha}} \log \P\Bigl(\forall j\leq \ell:\, X_{[j:n]} > \kappa n^\gamma\Bigr) 
	\leq- \ell \kappa ^\alpha. 
\end{equation}
Similarly, 
\begin{equation} \label{eq:ub2}
	\limsup_{n\to \infty} \frac{1}{n^{\gamma \alpha}} \log \P(M_n > m n^\gamma) \leq - m^\alpha. 
\end{equation} 
Turning to the first term on the right side of~\eqref{eq:A3}:

\begin{lemma}  \label{lem:ubmain}
	Let $\mathcal A_{m,\ell}^{(\ell)} \subset \R_+^m$ be the image of $\mathcal A_{m,\ell}$ under the projection $(x_j)_{j\in \N} \mapsto (x_1,\ldots,x_\ell)$. Let 
	\[
		I_{m,\ell} = \inf_{(x_1,\ldots, x_\ell) \in \mathcal A_{m,\ell}^{(\ell)}} \Bigl( \sum_{j=1}^\ell x_j^\alpha + \frac{1}{2\sigma^2} \Bigl( s - \sum_{j=1}^\ell x_j\Bigr)^2 \Bigr). 
	\]
	Then for every $m>0$ and $\ell \in \N$, 
	\[
		\limsup_{n\to \infty} \frac{1}{n^{\gamma\alpha}} \log \P\Bigl( \Bigl( \frac 1{n^\gamma} X_{[j:n]}\Bigr)_{j\in \N}\in \mathcal A_{m,\ell}, S_n =  \mu n + s_n n^\gamma  \Bigr) \leq - I_{m,\ell}.
	\] 
\end{lemma} 

\begin{proof}
Clearly
\begin{multline*}
	\P\Bigl( \Bigl( \frac{X_{[j:n]}}{n^\gamma}\Bigr)_{j\in \N}\in \mathcal A_{m,\ell}, S_n = \mu n + s_n n^\gamma \Bigr)\\
	 \leq \binom{n}{\ell} \sum_{m_1^{(n)},\ldots, m_\ell^{(n)}} \P(\forall j\leq \ell:\, X_j=m_j^{(n)}) \\
	 	 \times \P\Bigl(M_{n-\ell} \leq \kappa n^\gamma,\, S_{n-\ell} =\mu n+ s_n n^\gamma - \sum_{j=1}^\ell m_j^{(n)}\Bigr).
\end{multline*} 
The sum is over finite decreasing sequences $m_1^{(n)}\geq \cdots \geq m_\ell^{(n)}$ with $m\geq m_1^{(n)}$ for which there exist $y_{\ell+1},y_{\ell+2},\ldots$ such that 
\[
	\bigl( n^{-\gamma}m_1^{(n)}, \ldots,  n^{-\gamma}m_\ell^{(n)},y_1,y_2,\ldots\bigr) \in \mathcal A_{m,\ell}
\] 
or equivalently, $(n^{-\gamma} m_1^{(n)}, \ldots,  n^{-\gamma}m_\ell^{(n)}) \in \mathcal A_{m,\ell}^{(\ell)}$. 
By Proposition~\ref{prop:normal}, as $n\to \infty$ at fixed $\ell$,
\begin{multline*}
	\P\Bigl(M_{n-\ell} \leq \kappa n^\gamma,\, S_{n-\ell} =\mu n+ s_n n^\gamma - \sum_{j=1}^\ell m_j^{(n)}\Bigr) \\
	\leq \exp\Bigl( - \frac 1{2\sigma^2}\Bigl(s_n- \sum_{j=1}^\ell m_j^{(n)}\Bigr)^2 n^{\gamma\alpha}  + o(n^{\gamma\alpha})\Biggr)
\end{multline*} 
uniformly in $m_1^{(n)}+\cdots + m_\ell^{(n)} \leq \mathrm{const} n^\gamma$. It follows that 
\begin{multline*} 
	\P\Bigl( \Bigl( \frac{X_{[j:n]}}{n^\gamma}\Bigr)_{j\in \N}\in \mathcal A_{m,\ell}, S_n = \mu n + s_n n^\gamma \Bigr)\\
	\leq \e^{o(n^{\gamma\alpha})} \max_{(m_1^{(n)},\ldots, m_\ell^{(n)})} \exp\Biggl( - \sum_{j=1}^\ell m_j^\alpha  - \frac 1{2\sigma^2}\Bigl(s_n n^\gamma - \sum_{j=1}^\ell m_j^{(n)}\Bigr)^2\Biggr)
\end{multline*}
and then 
\begin{multline*}
	\limsup_{n\to \infty} \frac{1}{n^{\gamma\alpha}} \log 	\P\Bigl( \Bigl( \frac{X_{[j:n]}}{n^\gamma}\Bigr)_{j\in \N}\in \mathcal A_{m,\ell}, S_n = \mu n + s_n n^\gamma \Bigr)\\
	\leq - \inf_{(x_1,\ldots, x_\ell) \in \mathcal A^{(\ell)}_{m,\ell} } \Bigl( \sum_{j=1}^\ell x_j^\alpha + \frac{1}{2\sigma^2} \Bigl(s - \sum_{j=1}^\ell x_j\Bigr)^2 \Bigr) =  - I_{m,\ell}.
\end{multline*} 	
\end{proof} 

\begin{lemma} \label{lem:infconv}
	$\liminf_{\ell \to \infty} I_{m,\ell} \geq \inf_\mathcal A F_s$. 
\end{lemma} 

\begin{proof}
Let $(x_1^{(\ell)},\ldots, x_\ell^{(\ell)})\in \mathcal A^{(\ell)}_{m,\ell}$ be such that 
\[
	F_s(x_1^{(\ell)},\ldots, \ldots, x_\ell^{(\ell)},0,0,\ldots)\leq I_{m,\ell} + \frac 1\ell. 
\] 
Since $[0,m]^\N$ is compact in the product topology, upon passing to a subsequence we may assume that $\vect x^{(\ell)}= (x_1^{(\ell)},\ldots, x_\ell^{(\ell)},0,0,\ldots)$ converges in the product topology to some limiting sequence $\vect y \in \mathcal D$ with $y_1\leq m$. By the lower semicontinuity of $F_s$, 
\[
	F_s(\vect y) \leq \liminf_{\ell\to \infty} I_{m,\ell}.
\] 
By the definition of $\mathcal A_{m,\ell}^{(\ell)}$, there exists $\vect y^{(\ell)} \in \mathcal A_{m,\ell}$ such that $x_j^{(\ell)} = y_j^{(\ell)}$ for all $j\leq\ell$. The sequence $(\vect y^{(\ell)})$ converges to the same limiting $\vect y$ and therefore, because $\mathcal A_{m,\ell}$ is closed, we must have  $\vect y \in \mathcal A_{m,\ell}$. It follows that
\begin{equation*} 
   F_s(\vect y) \geq \inf_{\mathcal A_{m,\ell}} F_s \geq \inf_{\mathcal A} F_s.  \qedhere
\end{equation*} 
\end{proof} 

\begin{proof} [Proof of Proposition~\ref{prop:ub}] 
	Eqs.~\eqref{eq:A3}, \eqref{eq:ub1}, \eqref{eq:ub2} and Lemma~\ref{lem:ubmain} yield 
	\begin{equation} \label{eq:ub3} 
		\limsup_{n\to \infty} \frac{1}{n^{\gamma\alpha}} \log \P\Bigl( \Bigl( \frac{X_{[j:n]}}{n^\gamma} \Bigr)_{j\in \N} \in \mathcal A,\ S_n = \mu n + s_n n^\gamma \Bigr) \leq - \min\bigl(m^\alpha, \ell \kappa^\alpha, I_{m,\ell}\bigr)
	\end{equation}
	with $\min(a,b,c) = \min (a,b)$ if $c=\infty$ and $a,b<\infty$. The inequality holds true for all $\ell\in \N$ and  $m>0$. 
	
	If $\inf_{\mathcal A} F_s<\infty$, given $\eps>0$, choose $\ell_0$ and $m>0$ such that $m^\alpha$ and $\kappa \ell_0^\alpha$ are strictly larger than $\inf_{\mathcal A} F_s + \eps$. By Lemma~\ref{lem:infconv}, there exists $\ell_1\geq \ell_0$ such that 
	\[	
		 \inf_{\mathcal A} F_s - \eps \leq 
	   I_{m,\ell} \leq \inf_{\mathcal A} F_s +\eps< \min (m^\alpha, \kappa \ell_1^\alpha).
	\]
	The right side of Eq.~\eqref{eq:ub3} at $\ell =\ell_1$ is then $\leq - I_{m,\ell_1} \leq - \inf_{\mathcal A} F_s + \eps$ and we conclude by letting $\eps\downarrow 0$. 
	
	If $\inf_{\mathcal A} F_s = \infty$: given $M>0$, choose first $m$ so that $m^\alpha \geq M$ and $\ell_0 \in \N$ with $\kappa \ell_0^\alpha \geq M$. By Lemma~\ref{lem:infconv} there exists $\ell_1\geq \ell_0$  so that $I_{m,\ell_1} \geq M$ as well. The right side of Eq.~\eqref{eq:ub3} is smaller than $-M$ and we conclude by letting $M\to \infty$. 
\end{proof}

\begin{proof}[Proof of Theorem~\ref{thm:main}]
	The rate function $I_s = F_s - \inf_{\mathcal D} F_s$ is lower semicontinuous with good level sets by Lemma~\ref{lem:lsc}.
	Propositions~\ref{prop:lb} and~\ref{prop:ub} applied to $\mathcal A = \mathcal O = \mathcal D$ yield 
	\begin{equation}\label{eq:sna}
		\lim_{n\to \infty} \frac{1}{n^{\gamma\alpha}} \log \P( S_n = \mu n + s_n n^{\gamma}) = - \inf_{\mathcal D} F_s.
	\end{equation} 
	Eqs.~\eqref{eq:ldp-ub} and~\eqref{eq:ldp-lb} follow from Eq.~\eqref{eq:sna} and Propositions~\ref{prop:lb} and~\ref{prop:ub}. This completes the proof of the theorem. 
\end{proof} 

\subsection{Proof of Theorem~\ref{thm:dyntrans}} \label{sec:dyntrans}

The proof of Theorem~\ref{thm:dyntrans} builds on elementary properties of the function $g_s$, proven in Appendix~\ref{app:variational}: For $s>s_1$, the derivative $g'_s(y)$ has exactly two zeros $0<y_1(s) < y_2(s) < s$, and $g_s$ is increasing on $[0,y_1(s)]$, decreasing on $[y_1(s), y_2(s)]$ and then again increasing on $[y_2(s),\infty)$. The global minimum, for $s>s_1$, is at $y= y_2(s)$.

\begin{lemma}\label{lem:var2}
	Let $s\in \R$ and $y\geq 0$. Then $f_s(y) \leq g_s(y)$ with equality if and only if $z=0$ minimizes $g_{s-y}(z)$ on $[0,y]$. 
\end{lemma} 

\begin{proof} 
By the recursive equality~\eqref{eq:inf-convolution},
\[
	f_s(y) = y^\alpha + \inf_{z\in [0,y]} f_{s-y}(z) \leq y^\alpha + f_{s-y}(0) = y^\alpha + \frac{1}{2\sigma^2} (s-y)^2 = g_s(y).
\] 
Inserting $f_{s-y}\leq g_{s-y}$ we also get 
\[
	f_s(y) \leq y^\alpha + \inf_{z\in [0,y]} g_{s-y}(z) \leq y^\alpha + g_{s-y}(0) = g_s(y) 
\] 
and $f_s(y) \leq g_s(y)$ with equality if and only if $\inf_{z\in [0,y]} g_{s-y}(z) = g_{s-y}(0)$. 
\end{proof}

Lemma~\ref{lem:var2} leaves two possible scenarios for $f_s(y)= g_s(y)$: either $s-y\leq s_1$ and $z=0$ is in fact a global minimizer of $g_{s-y}(z)$; or $s-y>s_1$ but the cutoff $z\leq y$ makes the global maximum $z=y_2(s-y)$ of $g_{s-y}$ inaccessible. This leads to the following lemma. 

\begin{lemma} \label{lem:var3}
	Let $y_0(s):= \inf \{s>0: \, g_s(y) = g_s(0)\}$, with $\inf \varnothing = \infty$. Then, $f_s(y) < g_s(y)$ if and only if $s>s_1$ and $y_0(s-y) < y < s-s_1$. 
\end{lemma} 

\begin{proof} 
	By Lemma~\ref{lem:var2}, $f_s(y)< g_s(y)$ if and only if $\inf_{z\in [0,y]} g_{s-y}(z) < g_{s-y}(0)$. If this is the case, then necessarily $s-y>s_1$ (because $z=0$ is not a global minimizer of $g_{s-y}$) and, as $g_{s-y} \geq g_{s-y}(0)$ on $[0, y_0(s-y)]$, hence $y> y_0(s-y)$. Conversely, if $y_0(s-y) < y< s-s_1$, then we may pick $z= y_0(s-y)+\eps$ with $\eps$ sufficiently small and have $z\leq y$ and $g_{s-y}(z)< g_{s-y}(0)$. 
\end{proof} 

\begin{proof}[Proof of Theorem~\ref{thm:dyntrans}]
	Set 
	\[
		s_2:= \inf\{ s >s_1:\ \exists y \geq 0:\, \inf_{z\in [0,y]} g_{s-y}(z) < g_{s-y}(0)\}.
	\] 
	The difference 
	\begin{equation}\label{eq:gtdiff}
		g_t(z) - g_t(0) = z^\alpha + \frac{1}{2\sigma^2}\bigl( - 2 t z + z^2\bigr)
	\end{equation} 
	goes to  $-\infty$ as $t\to \infty$ at fixed $y$ and $z>0$, therefore $g_{s-y}(z)<g_s(0)$ for all sufficiently large $s$. It follows that $s_2<\infty$. 

	If $s>s_2$ then by Lemma~\ref{lem:var3}, $s-s_1> y > y_0(s-y)$. Now, $t\mapsto y_0(t)$ is decreasing therefore $y_0(s-y) \geq y_0(s)$ and $s-s_1 > y_0(s)$. But, as $s\downarrow s_1$, $y_0(s)\geq y_1(s)\to y_*>0$ (with $y_*$ defined in Appendix~\ref{app:variational}), so the inequality $s-s_1 \geq y_0(s)$ cannot hold true for $s$ close to $s_1$, hence $s_2>s_1$. 

	Furthermore, the difference~\eqref{eq:gtdiff} decreases as $t$ increases. Therefore, if $g_{s-y}(z)< g_{s-y}(0)$ for some $z\leq y$, then this holds true as well for all $s'\geq s$. It follows that $\inf_{z\in [0,y]} g_{s-y}(z) < g_{s-y}(0)$ for all $s>s_2$, and $\inf_{z\in [0,y]} g_{s-y}(z) = g_s(0)$ for all $s< s_2$ and then, by continuity, for $s=s_2$. Lemma~\ref{lem:var2} yields $f_s=g_s$ on $\R_+$ if and only if $s>s_2$. Finally, for $s>s_2$, the set $J_s$ of $y$'s with $f_s(y) < g_s(y)$ is equal to 
	\[
		J_s = \bigl\{ y:\, y_0(s-y) < y < s-s_1\bigr\} \subset \bigl(y_0(s), s-s_1\bigr)
	\] 
	with $y_0(s)>0$. 		
\end{proof} 

\appendix 

\section{Auxiliary variational problem} \label{app:variational}

Here we analyze the function $g_s(y) = y^\alpha + \frac1{2\sigma^2} (s-y)^2$. The properties proven here are implicit in Nagaev \cite{nagaev1968}, see also Lemma~2.4 and 2.14 as well as Fig.~1 in  Ercolani, Jansen, Ueltschi \cite{ercolani-jansen-ueltschi2019}. We prove the properties in detail for the reader's convenience. 

Clearly $g_s$ is increasing on $[s,\infty)$, thus we only need to investigate $s>0$ and $y\in [0,s]$. The first and second derivatives are 
\[
	g_s'(y) = \alpha y^{\alpha-1} - \frac{s-y}{\sigma^2},\quad g''_s(y) = - \alpha(1-\alpha) y^{\alpha-2} +\frac 1{\sigma^2}. 
\] 	
Set $y_*:= ( (1-\alpha)\alpha \sigma^2)^\gamma$. Then $g''_s(y_*)=0$ and $g'_s(y)$ is decreasing on $[0,y_*]$ and increasing on $[y_*,\infty)$. 
At $0$ and $\infty$ the derivative goes to $\infty$. The value at $y_*$ is strictly negative if and only if 
\[
	s > \alpha \sigma^2 y_*^{\alpha-1} + y_* 
	 = \frac{2-\alpha}{1-\alpha} \bigl( (1-\alpha)\alpha  \sigma^2 \bigr)^\gamma 
	 = \frac 1 \gamma (\alpha \sigma^2)^\gamma (1- \alpha)^{\gamma-1} =:s_0.
\] 
Thus, if $s\leq s_0$ then $g'_s(y_*) \geq 0$ and $g_s$ is increasing on $[0,\infty)$. If $s>s_0$, the derivative has exactly two zeros, we label them as 
\[
	0 < y_1(s) < y_* < y_2(s) < s. 
\] 
The first zero is a local maximizer and the second zero is a local minimizer. Thus, the global minimum of $g_s$ is $g_s(0)$ or $g_s(y_2(s))$. The difference between the two is 
\begin{align*}
	g_s(y_2(s)) - g_s(0)& = y_2(s)\Bigl( y_2(s)^{\alpha-1} + \frac{1}{2\sigma^2}(y_2(s) - 2s) \Bigr)  \\
	& = y_2(s)\Bigl( \frac{1}{\alpha}\frac{s-y_2(s)}{\sigma^2} + \frac{1}{2\sigma^2}(y_2(s) - 2s) \Bigr)\\
	& = \frac{y_2(s)}{2 \sigma^2 \alpha}\Bigl( (2-\alpha)(s- y_2(s)) - \alpha s\Bigr). 
\end{align*}
Hence, the two local minima are equal if and only if $s - y_2(s) = \alpha s/(2-\alpha)$ and $y_2(s) = (2-2\alpha) \gamma s $. The equation $g'_s(y_2(s)) =0$ yields 
\[
	\alpha \Bigl( 2 (1-\alpha) \gamma s\Bigr)^{\alpha-1} = \frac{\alpha \gamma s}{\sigma^2}
\]
and 
\[
	s = \frac 1\gamma (\sigma^2)^\gamma ( 2-2\alpha)^{\gamma-1} =s_1. 
\] 
Notice that $\gamma < 1$ and 
\[
	\frac{s_1}{s_0} = \frac{2^{\gamma-1}}{\alpha^\gamma} = \frac 1 2 \Bigl( \frac{2}{2- \gamma^{-1}}\Bigr)^\gamma = \frac 12 \Bigl( 1- \frac{1}{2\gamma}\Bigr)^{-\gamma}> \frac 12\Bigl( 1- \frac 1 2\Bigr)^{-1} = 1,
\]
so we have indeed $s_1>s_0$. The difference $g_s(y_2(s))- g_s(0)$ is decreasing on $(s_0,\infty)$ because 
\begin{align*}
	\frac{\dd}{\dd s} \bigl( g_s(y_2(s)) - g_s(0)\bigr) & = g'_s(y_2(s)) y'_2(s) + \frac{1}{\sigma^2}\Bigl( \bigl(s-y_2(s)\bigr) - s\Bigr)\\
	& = g'(y_2(s)) y'_2(s) - \frac{y_2(s)}{\sigma^2}\\
	& = -  \frac{y_2(s)}{\sigma^2}< 0. 
\end{align*}
The difference vanishes if and only if $s=s_1$. Thus, we have shown that 
 $g_s(y_2(s)) <g_s(0)$ if and only if $s>s_1$.

\subsubsection*{Acknowledgments.} The project started  within the DFG Research Training Network GRK 2123 \emph{High-dimensional phenomena in probability -- fluctuations and discontinuity}. I thank Sander Dommers for helpful discussions during the initial stages of the project at Ruhr-Universit{\"a}t Bochum. 


\providecommand{\bysame}{\leavevmode\hbox to3em{\hrulefill}\thinspace}
\providecommand{\MR}{\relax\ifhmode\unskip\space\fi MR }
\providecommand{\MRhref}[2]{%
  \href{http://www.ams.org/mathscinet-getitem?mr=#1}{#2}
}
\providecommand{\href}[2]{#2}

\end{document}